\documentclass[12pt,reqno]{article}
\usepackage{amssymb,amscd,amsmath,amsthm,color}
\newtheorem{theorem}{Theorem}[section]

\newtheorem{cor}[theorem]{Corollary}
\newtheorem{prop}[theorem]{Proposition}

\usepackage{enumerate,amssymb}
\theoremstyle{definition}

\newtheorem{question}[theorem]{Question}
\theoremstyle{remark}
\newtheorem{remark}[theorem]{Remark}

\numberwithin{equation}{section}

\DeclareMathOperator{\cotan}{cotan}
\DeclareMathOperator{\arccotan}{arccot}

\def\dim{\mathrm{span\,}}
\def\dim{\mathrm{dim\,}}
\def\eps{\varepsilon}
\def\rk{\mathrm{rank\,}}
\def\eps{\varepsilon}

\def\bC{\mathbb{C}}

\def\bM{\mathbb{M}}
\def\bN{\mathbb{N}}
\def\bR{\mathbb{R}}

\def\bZ{\mathbb{Z}}

\def\trE{\mathcal{E}}
\def\trF{\mathcal{F}}

\def\trH{\mathcal{H}}

\def\trP{\mathcal{P}}
\def\trQ{\mathcal{Q}}
\def\trR{\mathcal{R}}
\def\trS{\mathcal{S}}

\begin{document}
\baselineskip=15pt

\title{Involutions and   angles between subspaces }

\author{Jean-Christophe Bourin and  Eun-Young Lee  }

\date{ }

\maketitle

\vskip 10pt\noindent
{\small 
{\bf Abstract.}  If $S$ is an involution matrix, which means $S^2=I$, then
$S$ is unitarily equivalent to 
$$(\eps I_k)\oplus\left\{\bigoplus_{i=1}^{m}
\begin{pmatrix}0&x_i^{-1}\\ x_i&0\end{pmatrix}\right\}$$
where $\eps=\pm 1$, $I_k$ is the identity of size $k=|{\mathrm{Tr\,}}S|$, and $0<x_i\le 1$, $i=1,\ldots,m=n-k$. This can be used to introduce principal angles between subspaces and yields   formulas for the  angles  between the eigenspaces of $S$ such as $\sin \alpha_i=2/(x_i+x_i^{-1})$. 

}

\section*{}

In Section 2, we will show our main results related to the structure of involutory matrices, or involutions. This refines some recent observations by Ikramov and this yields  simple formulas for the principal angles between the two subspaces associated to an involution. 
Section 3 studies dilations into involutions.

From our structural result for involutions of Section 2, one might easily derive the classical notion of principal angle. 
However, for convenience of the reader and completeness,  we recall  in  Section 1 some facts on principal angles, with detailed proofs. 

\section{Principal angles}

Let $\bM_n$ denote either the space of $n\times n$ complex or real matrices. In the case of real matrices, a unitary matrix is   an orthogonal matrix. We freely identify an operator $A$ on $\trH=\bC^n$ or $\bR^n$ with its matrix $A$ in the canonical basis. Given $A\in\bM_n$ with rank $r$, the polar decomposition says that there exist two orthonormal basis $\{x_i\}$ and  $\{y_i\}$ of $\trH$  such that, using Dirac notation, 
$$
A=\sum_{i=1}^n a_i |y_i\rangle\langle x_i|
$$
for some scalars $a_1\ge \cdots \ge a_n$, the  singular values of $A$. If $A$ is the product of two orthoprojections (i.e., orthogonal projections)  this leads to the notion of principal angles.

\newpage

\vskip 5pt
\begin{prop}\label{prop1} Let $P,Q\in\bM_n$ be two nonzero orthoprojections,  let $\mu_1\ge \cdots\ge \mu_n$ be the singular values of $PQ$, and let $r=\min\{\rk P,\rk Q\}$. Then there exists an orthonormal system $\{q_i\}_{i=1}^r$ in the range $\trQ$ of $Q$, and an orthonormal system $\{p_i\}_{i=1}^r$ in the range $\trP$ of $P$, such that
$$
PQ=\sum_{i=1}^r \mu_i |p_i\rangle\langle q_i|.
$$
Furthermore we necessarily have: 
 \begin{itemize}
 \item[(a)] $i\neq j \Rightarrow \langle p_i, q_j\rangle =0$;
 \item[(b)]  $\langle p_i, q_i\rangle =\mu_i$;
  \item[(c)] $\trH_c:=\trQ\ominus{\mathrm{span}}\{ q_i, 1\le i\le r\}$ is orthogonal to $\trP$. \end{itemize}
The principal angles $0\le \alpha_i^{\uparrow}\le \cdots\le \alpha_r^{\uparrow}\le \pi/2$ between $\trP$ and $\trQ$ (or between $\trQ$ and $\trP$) are defined by 
$$\alpha_i^{\uparrow}=\arccos \langle p_i,q_i\rangle, \qquad i=1,\ldots, r.$$
\end{prop}

\vskip 5pt
\begin{proof} The first part of the proposition immediately follows from the polar decomposition of $PQ$, since $PQ$ has rank $r'\le r$, its support $\trS\subset\trQ$ and  its range  $\trR\subset\trP$.

The assertion (c) is immediate from $\trH_c\subset\ker P$. To prove (a) and (b), observe that
$$
PQ= \left(\sum_{i=1}^{r'} |p_i\rangle\langle p_i|\right) PQ  \left(\sum_{j=1}^{r'} |q_j\rangle\langle q_j|\right) 
= \sum_{i,j=1}^{r'}\langle p_i, q_j\rangle |p_i\rangle\langle q_j|,
$$
hence,
$$
\sum_{i,j=1}^{r'}\langle p_i, q_j\rangle |p_i\rangle\langle q_j| =\sum_{i=1}^{r'} \mu_i |p_i\rangle\langle q_i|.
$$
Since the operators $|p_i\rangle\langle q_j|$, $1\le i,j\le r'$, form a basis of the operators from $\trS$  to $\trR$,
it follows that
$$
\langle p_i,q_j\rangle = \delta_i^j \mu_i, \qquad 1\le i, j\le r'.
$$ 
Next, if $r'<i\le r$ then $p_i$ is in the kernel of $PQ$ and so is orthogonal to any $q_j$ and we also have $\mu_i=\langle p_i,q_i\rangle $. 

 As $\mu_i\ge 0$, we do have $\pi/2\ge \arccos\langle p_i,q_i\rangle\ge 0$. Lastly, as 
$$
QP= \sum_{i=1}^{r'} \mu_i |q_i\rangle\langle p_i|,
$$
the principal angles between $\trQ$ and $\trP$ are the same than those between $\trP$ and $\trQ$.
\end{proof}

\vskip 5pt
As a byproduct of Proposition \ref{prop1} we have  a complete description of the relative position of two subpsaces through the following joint reduction of their orthoprojections.

\vskip 5pt
\begin{cor}\label{corbd} With  notations of Proposition \ref{prop1}, assuming $r=\dim\trP\le \dim \trQ$,
let  $\{1,2,\ldots,r\}$ be partitioned with $J_1:=\{i : \mu_i=1\}$ and $J_{<1}:=\{i : \mu_i<1\}$, and let
$$
\trH_{=1}:={\mathrm{span}}\{q_i : i\in J_1\}= \trP\cap\trQ, \qquad
 \trH_{<1}:=\bigoplus_{i\in  J_{<1}}{\mathrm{span}}\{q_i, p_i\} .
$$
Further,  consider 
$\trH_c:=\trQ\ominus (\trH_{=1}\oplus  \trH_{<1})$ and   $\trH_0:=(\trP+\trQ)^{\perp}$.
Then with respect to the orthogonal decomposition
$$
\trH= \trH_{=1}\oplus \left\{\bigoplus_{i\in  J_{<1}}{\mathrm{span}}\{q_i, p_i\}\right\}\oplus \trH_c \oplus \trH_0
$$
and   orthonormal basis $\{q_i,q_i^0\}$ for ${\mathrm{span}}\{q_i, p_i\}$, $j\in J_{<1}$, we have the block-diagonalizations,
\begin{equation}\label{jd1}
Q=I\oplus \left\{\oplus^{|J_1|} \begin{pmatrix} 1&0 \\ 0&0\end{pmatrix}\right\}\oplus I\oplus 0,
\end{equation}
\begin{equation}\label{jd2}
P=I\oplus \left\{\bigoplus_{i\in J_{<1}} \begin{pmatrix} \cos^2 \alpha_i^{\uparrow}&\cos \alpha_i^{\uparrow}\sin \alpha_i^{\uparrow} \\ \cos \alpha_i^{\uparrow}\sin \alpha_i^{\uparrow}&\sin^2 \alpha_i^{\uparrow}\end{pmatrix}\right\}\oplus 0\oplus 0.
\end{equation}
Conversely, if these two joint block diagonalizations hold, then the principal angles between $\trP$ and $\trQ$ are $\alpha_i^{\uparrow}$, $i=1,\ldots, r$.
\end{cor}

\vskip 5pt

 Corollary \ref{corbd} is a  direct consequence of Proposition \ref{prop1} with $\dim\trP\le \dim \trQ$. Indeed, Proposition \ref{prop1} then says that there is an orthonormal basis of $\trH$,
$$
\{q_i\}_{i\in J_{=1}}\sqcup\{q_i,q_i^0\}_{i\in J_{<1}}\sqcup\{h_i\}_{i\in\sigma}
$$
such that, for some (possibly) vacuous  set of indices $\sigma$,
\begin{equation}\label{J1}\trQ=\bigoplus_{i\in J_{=1}}\bC q_i \oplus \bigoplus_{i\in J_{<1}}\bC q_i \oplus \bigoplus_{i\in\sigma}\bC h_i,
\end{equation}
\begin{equation}\label{J2}\trP=\bigoplus_{i\in J_{=1}}\bC q_i \oplus \bigoplus_{i\in J_{<1}}\bC\left( (\cos\alpha_i^{\uparrow})q_i +(\sin\alpha_i^{\uparrow})q_i^0\right).
\end{equation}
Since
$$
\begin{pmatrix} \cos\alpha_i^{\uparrow} 
\\ \sin\alpha_i^{\uparrow} \end{pmatrix}
\begin{pmatrix} \cos\alpha_i^{\uparrow} 
& \sin\alpha_i^{\uparrow} \end{pmatrix}=
 \begin{pmatrix} \cos^2 \alpha_i^{\uparrow} &\cos \alpha_i^{\uparrow}\sin \alpha_i^{\uparrow} \\ \cos \alpha_i^{\uparrow}\sin \alpha_i^{\uparrow} &\sin^2 \alpha_i^{\uparrow}\end{pmatrix},
$$
Corollary \ref{corbd} follows. Here is another, more detailed and computational proof.

\vskip 5pt
\begin{proof} The orthogonality conditions  (a) and (c) in Proposition \ref{prop1} shows the orthogonal decomposition of $\trH$. For $i\in J_{<1}$, we have $p_i\in\mathrm{span}\{q_i,p_i\}$ and $\langle p_i,q_i\rangle=\cos \alpha^{\uparrow}_i>0$. So there exist an orthonormal basis $\{q_i,q_i^0\}$ of $\trS_i:=\mathrm{span}\{q_i,p_i\}$ such that 
$$
p_i=(\cos \alpha^{\uparrow}_i)q_i+(\sin \alpha^{\uparrow}_i)q^0_i.
$$
Since $Qq_i=q_i$ and $Qp_i=QPp_i=(\cos \alpha^{\uparrow}_i)q_i$, we infer $Q((\sin \alpha^{\uparrow}_i)q^0_i)=0$ so $Qq^0_i=0$. This yields the block diagonalization for $Q$.
 To get the block diagonalization for $P$, recall that $\trH_c\perp\trP$, and for $i\in J_{<1}$,
$$
P(q_i)=PQ (q_i) = (\cos \alpha^{\uparrow}_i)p_i=(\cos^2 \alpha^{\uparrow}_i)q_i+(\cos\alpha_i^{\uparrow}\sin \alpha^{\uparrow}_i)q^0_i,
$$
yielding the first column of the matrix representing $P$ on $\trS_i$. Next, 
$$
P((\sin \alpha^{\uparrow}_i)q^0_i)=P(p_i-(\cos \alpha^{\uparrow}_i)q_i)=p_i- (\cos^2 \alpha^{\uparrow}_i)p_i=(\sin^2\alpha^{\uparrow}_i)p_i
$$
so
$$
P(q^{0}_i)=(\sin\alpha^{\uparrow}_i)p_i=(\sin\alpha^{\uparrow}_i\cos \alpha^{\uparrow}_i)q_i+(\sin^2 \alpha^{\uparrow}_i)q^0_i
$$
completing the representation of the action of $P$ on $\trS_i$.

To get the last assertion, observe that for $i\in J_{<1}$ the nonzero singular value  of 
$$
 \begin{pmatrix} \cos^2 \alpha_i^{\uparrow}&\cos \alpha_i^{\uparrow}\sin \alpha_i^{\uparrow} \\ \cos \alpha_i^{\uparrow}\sin \alpha_i^{\uparrow}&\sin^2 \alpha_i^{\uparrow}\end{pmatrix}
\begin{pmatrix} 1&0\\ 0&0
\end{pmatrix}
$$
is $\sqrt{\cos^4\alpha_i^{\uparrow} +  \cos^2 \alpha_i^{\uparrow}\sin^2 \alpha_i^{\uparrow}}= \cos \alpha_i^{\uparrow}$.
\end{proof}

\vskip 5pt
Needless to say, principal angles are invariant under the action of a unitary operator.

\vskip 5pt
\begin{cor} Let $(\trP,\trQ)$ and $(\trP',\trQ')$ be two pairs of  subspaces in $\bC^n$ or $\bR^n$ such that 
$r:=\dim\trP=\dim\trP'\le\dim\trQ=\dim\trQ'.$
Then the following conditions are equivalent :
\begin{itemize}
\item[(a)] $\alpha_i^{\uparrow}(\trP,\trQ)=\alpha_i^{\uparrow}(\trP',\trQ')$ for all $i=1,\cdots,r$.
\item[(b)] $U(\trP)=\trP'$ and $U(\trQ)=\trQ'$ for some unitary $U\in\bM_n$.
\end{itemize}
\end{cor}

\vskip 5pt
\begin{proof} (b) says that $U$ induces an isometric operator from $\trP$ to $\trP'$. Thus $U^*$ induces the inverse isometry from $\trP'$ to $\trP$. Therefore the respective projections $P,P'$ on $\trP,\trP'$ satisfy $P'=UPU^*$. Similary $Q'=UQU^*$. Hence $PQ$ and $P'Q'$ have the same singular values, so $(a)$ holds.

Conversely, if (a) holds, then the previous corollary shows that $P'=UPU^*$ and $Q'=UQU^*$
for some unitary $U$, and so  $U(\trP)=\trP'$ and $U(\trQ)=\trQ'$.
\end{proof}

\vskip 5pt
From the joint block-diagonalizations of Corollary \ref{corbd}, we may derive several facts, for instance on the spectra of $P\pm Q$ or $PQ\pm QP$, by  focusing on two-by-two matrices.
Some examples are given in the next corollaries.

\vskip 5pt
\begin{cor}\label{corDT}
Let $P,Q\in\bM_n$ be two nonzero orthoprojections. Then
$$
\| P+Q\|_{\infty} =1+\| PQ\|_{\infty}.
$$
\end{cor}

\vskip 5pt
\begin{proof} If $\trP$ and $\trQ$ intersect, or if $\trP$ and $\trQ$ are orthogonal, the result is trivial.  Indeed (1) if $\trP$ and $\trQ$ are orthogonal then $PQ=0$ and $P+Q$ is an orthoprojection (and we have $1=1$ in the corollary); (2) if $\trP$ and $\trQ$ intersect, pick a unit vector $x\in \trP\cap\trQ$, and note that $\|(P+Q)x\|= 2 =1+\|PQx\|$
(and we have $2=2$ in the corollary).

For the other cases, we use the joint diagonalization of Corollary \ref{corbd} for the respective orthoprojections $P$ and $Q$  onto $\trP$ and $\trQ$. Computing the characteristic polynomial of
$$
 \begin{pmatrix} 1+ \cos^2 \alpha_i^{\uparrow}&\cos \alpha_i^{\uparrow}\sin \alpha_i^{\uparrow} \\ \cos \alpha_i^{\uparrow}\sin \alpha_i^{\uparrow}&\sin^2 \alpha_i^{\uparrow}\end{pmatrix}
$$
reveals that its largest root is $1+\cos\alpha_i^{\uparrow}$. Since $1+\cos\alpha_1^{\uparrow}=1+\|PQ\|_{\infty}$, the equality follows.
\end{proof}

\vskip 5pt
\begin{cor}\label{corgap}
Let $P,Q\in\bM_n$ be two orthoprojections of rank $k$ with respective ranges $\trP$ and $\trQ$. Then
$$
\| P-Q\|_{\infty} =\sin \alpha_k^{\uparrow}(\trP,\trQ).
$$
\end{cor}

\vskip 5pt
\begin{proof} We may argue as in the previous proof or simply note
that the blocks for $P-Q$ in the joint  block-diagonalization of $P$ and $Q$ are
$$
 \begin{pmatrix} 1- \cos^2 \alpha_i^{\uparrow}&-\cos \alpha_i^{\uparrow}\sin \alpha_i^{\uparrow} \\ -\cos \alpha_i^{\uparrow}\sin \alpha_i^{\uparrow}&-\sin^2 \alpha_i^{\uparrow}\end{pmatrix}
=\sin \alpha_i^{\uparrow}
 \begin{pmatrix} \sin \alpha_i^{\uparrow}&-\cos \alpha_i^{\uparrow} \\ -\cos \alpha_i^{\uparrow}&-\sin \alpha_i^{\uparrow}\end{pmatrix}
$$
with operator norms $\sin \alpha_i^{\uparrow}$. 
\end{proof}

\vskip 5pt
\begin{cor}\label{corperp}
Let $(\trP,\trQ)$  be a pair of  complementary subspaces in $\bC^n$ (or  $\bR^n$) with $r:=\dim \trP\le \dim\trQ$. Then, for all $j=1,\ldots, r$,
$$
\alpha_j^{\uparrow} (\trP, \trQ)= \alpha_j^{\uparrow} (\trP^{\perp}, \trQ^{\perp}).
$$
Hence,  there exists a unitary $U\in\bM_n$ such that $U(\trP)=\trQ^{\perp}$ and $U(\trQ)=\trP^{\perp}$.
\end{cor}

\vskip 5pt
\begin{proof} This can be derived from corollary \ref{corbd}, or equivalently from \eqref{J1} and \eqref{J2}. Since we have complementary subspaces, \eqref{J1} and \eqref{J2} reads as
\begin{equation*}\trQ= \bigoplus_{i\in J_{<1}}\bC q_i \oplus \bigoplus_{i\in\sigma}\bC h_i,
\qquad \trP= \bigoplus_{i\in J_{<1}}\bC\left( (\cos\alpha_i^{\uparrow})q_i +(\sin\alpha_i^{\uparrow})q_i^0\right)
\end{equation*}
and so
\begin{equation*}\trQ^{\perp}= \bigoplus_{i\in J_{<1}}\bC q_i^0,
\qquad \trP^{\perp}= \bigoplus_{i\in J_{<1}}\bC\left( (-\sin\alpha_i^{\uparrow})q_i +(\cos\alpha_i^{\uparrow})q_i^0\right) \oplus \bigoplus_{i\in\sigma}\bC h_i,
\end{equation*}
so that for all $j\in J_{<1}=\{1,\ldots,r\}$, we have $\alpha_j^{\uparrow} (\trP, \trQ)= \alpha_j^{\uparrow} (\trP^{\perp}, \trQ^{\perp})$.
\end{proof}

\vskip 5pt
\begin{remark} Principal angles (or canonical angles) have been introduced by Jordan in \cite{Jo}. Relations \eqref{J1}-\eqref{J2} are in this venerable article as well as Corollary \ref{corperp}. Corollary \ref{corbd} is sometimes associated to a 1983 work of Wedin (see \cite[p.\ 1420]{BoSp}). Principal angles are often introduced in the literature via the CS-decomposition, as in the survey \cite{BoSp}. Using the polar decomposition as in Proposition \ref{prop1} is  both simple and quite natural, since we may obviously expect that the relative position between
two subspaces is described by the singular values of the product of their projections $PQ$ (or the eigenvalues of $PQP$). Corollary \ref{corDT} is Duncan-Taylor's inequality \cite{DT}, for a remarkable recent generalization, see \cite{AC}.
The sine of the largest principal angle in Corollary \ref{corgap} is the gap between the two subspaces. A quite deep paper \cite{QLZ} develops much more metrics on the set of subspaces of a fixed dimension.
\end{remark}

\section{Structure of involutions}

An involution $S$ in the space $\bM_n$ of  square complex (or real) matrices of size $n$ is a matrix such that $S^2=I$, the identity. This is also called a symmetry. Then 
$$
E=(I+S)/2
$$
is a projection (the spectral projection of $S$ associated to $1$), and $F=I-E=(I-S)/2$ is the complementary projection (the spectral projection of $S$ associated to $-1$). Thus $S=E-F$ and letting $\trE$ and $\trF$ be the respective ranges of $E$ and $F$ one has the direct sum decomposition ($\trH$ still denotes $\bC^n$ or $\bR^n$),
$$
\trH =\trE \oplus \trF.
$$
We also use the notation $A\oplus B$ (and more generally $A_1\oplus\cdots\oplus A_m$) for the block diagonal matrices
of the form
$$
\begin{pmatrix} A&0 \\ 0&B
\end{pmatrix}, \quad A\in\bM_n, \ B\in \bM_{n'}.
$$

Given $X,Y\in\bM_n$, the symbol $X\simeq Y$ means $X=UYU^*$ for some unitary (orthogonal) matrix $U$. Any involution has a trace in $\bZ$. The following theorem shows the structure of an involution. 

\begin{theorem}\label{th-inv} Let $S\in\bM_n$ be an involution with ${\mathrm{Tr}}S=\eps k$ for some $k\in\bN$ and $\eps=\pm1$. Then, with   $m=(n-k)/2$,
$$
S\simeq(\eps I_k)\oplus\left\{\bigoplus_{i=1}^{m}\begin{pmatrix}0&x_i^{-1}\\ x_i& 0\end{pmatrix}\right\}\simeq(\eps I_k)\oplus\left\{\bigoplus_{i=1}^{m}
\begin{pmatrix}1&a_i\\ 0& -1\end{pmatrix}\right\}
$$
where    $0<x_1\le \ldots\le x_m\le 1$ and $a_1\ge \ldots\ge a_m\ge 0$,  and  there are exactly $m$  principal angles $\alpha_1^{\uparrow}\le \cdots\le \alpha_m^{\uparrow}$ between the eigenspaces of $S$, given by
$$
\alpha_i^{\uparrow} = \arcsin \frac{2}{x_i+x_i^{-1}}=2\arctan x_i=\arccotan(a_i/2).
$$
Moreover, $a_i=x_i^{-1}-x_i$ for all $i=1,\ldots,m$.
\end{theorem}

\vskip 5pt
We see that for $0<x\le 1$,
$$
2\arctan x= \arcsin \frac{2}{x+x^{-1}}.
$$
Despite of the symmetric form of the right hand side, this formula does not extend to $x>1$. In fact, it is well known that
$\arctan x + \arctan x^{-1}= \arctan x + \arccotan x = \pi/2$.

\vskip 5pt
\begin{proof} Let
$
E=(I+S)/2
$
and
let $\trE$ be the range of $E$.
We assume $0<d:=\dim \trE<n$, since the cases  $\dim\trE=0,n$ are trivial. Replacing $S$ by $-S$ if necessary, since
$$
-
\begin{pmatrix}0& x^{-1}\\ x& 0\end{pmatrix}\simeq 
\begin{pmatrix}0& x^{-1}\\ x& 0\end{pmatrix},
$$
we may also assume that ${\mathrm{Tr\,}}S\ge 0$, that is $\dim\trE\ge \dim\trF$ where $\trF$ is the 
range of the complementary projection $F=I-E$. Then
 $E$ has a block  matrix representation, with respect to an orthonormal basis of $\bC^n$ starting with a basis of $\trE$, of the form
$$
E\simeq\begin{pmatrix} I_d& R \\
0&0
\end{pmatrix} 
$$
where $R\in\bM_{d, n-d}$, $d\ge n-d$.   We have  $R=UDV$ with  unitary  $U\in\bM_d$, $V\in\bM_{n-d}$, and a rectangular diagonal matrix $D$ with the singular values of $R$ arranged in nondecreasing order down to the diagonal. The unitary congruence implemented by $U^*\oplus V$ shows that
$$
E\simeq \begin{pmatrix} I_d& D \\
0&0
\end{pmatrix}.
$$
Recall that $D$ is a rectangular matrix of the form
$$
D=\begin{pmatrix} \Delta \\ 0 \end{pmatrix}
$$
where $\Delta\in\bM_{n-d}$ is diagonal with the singular values of $R$ down to the diagonal and $0$ stands here for  the zero  $(2d-n)\times d$  matrix ( vacuous if $2d=n$). Thus, 
$$
E\simeq \begin{pmatrix} I_{n-d} & 0&  \Delta \\
0& I_{2d-n}&0\\
0&0&0
\end{pmatrix} \simeq \begin{pmatrix} I_{n-d} & \Delta&  0 \\
0& 0&0\\
0&0&I_{2d-n}
\end{pmatrix}.
$$
Since $S=2E-I$ we have
$$
S
\simeq \begin{pmatrix} I_{n-d} & \Delta&  0 \\
0& -I_{n-d}&0\\
0&0&I_{2d-n}
\end{pmatrix}.
$$
Hence
$$
S\simeq I_{2d-n} \oplus\left\{\bigoplus_{i=1}^m \begin{pmatrix} 1&a_i \\ 0 &-1\end{pmatrix}\right\}
$$
where $2d-n={\mathrm{Tr\,}}S$, $m=2d$, and $0\le a_1\le \cdots \le a_m $ are the diagonal entries of $\Delta$.
This proves the second decomposition of our theorem.

Let $a\ge 0$. To prove the first decomposition of the theorem, it suffices to show that 
$$
\begin{pmatrix} 1&a \\ 0 &-1\end{pmatrix}\simeq \begin{pmatrix} 0& x^{-1}\\ x&0\end{pmatrix} 
$$
for some $1\ge x>0$. Since these two real matrices have the same eigenvalues $\pm1$, by triangularization, this holds when the two matrices have the same Hilbert-Schmidt norm. Thus $x^{-2}+x^2=2+a^2$, hence $a=x^{-1}-x$, giving the last assertion of the theorem.

From the unitary congruence that we have just established for $S$, we infer that
we have two unitary congruences implemented by the same unitary matrix $V$,
$$
E=\frac{S+I}{2}=V \cdot  \left(\frac{(1+\eps) I_k}{2}\right)\oplus\left\{\frac{1}{2}\bigoplus_{i=1}^{m}
\begin{pmatrix}1&x_i^{-1}\\ x_i& 1\end{pmatrix}\right\}\cdot V^*
$$
and
$$
F=\frac{-S+I}{2}=V \cdot  \left(\frac{(-1+\eps) I_k}{2}\right)\oplus\left\{\frac{1}{2}\bigoplus_{i=1}^{m}
\begin{pmatrix}1&-x_i^{-1}\\ -x_i& 1\end{pmatrix}\right\}\cdot V^*.
$$
One of the matrix $(-1+\eps) I_k$ or $(1+\eps) I_k$ vanishes. Thus the principal angles between $\trE$ and $\trF$ are  the principal angles between the ranges of the two projections in $\bM_{2m}$, 
$$
E_0:=\frac{1}{2}\bigoplus_{i=1}^{m}
\begin{pmatrix}1&x_i^{-1}\\ x_i& 1\end{pmatrix}, \qquad
F_0:=\frac{1}{2}\bigoplus_{i=1}^{m}
\begin{pmatrix}1&-x_i^{-1}\\ -x_i& 1\end{pmatrix},
$$
These two projections have the same ranges, respectively, as the two rank $m$ orthogonal projections,
\begin{equation}\label{jd3}
E_1=\bigoplus_{i=1}^{m}
\frac{1}{1+x_i^2}\begin{pmatrix}1&x_i\\ x_i& x_i^2\end{pmatrix},\qquad
F_1=\bigoplus_{i=1}^{m}
\frac{1}{1+x_i^2}\begin{pmatrix}1&-x_i\\ -x_i& x_i^2\end{pmatrix}.
\end{equation}
We observe that $E_1F_1$ has $m$ nonzero singular values $\mu_i$ given by  
$$
\mu_i^2={\mathrm{Tr}\,}\frac{1}{(1+x_i^2)^2}\begin{pmatrix}1&x_i\\ x_i& x_i^2\end{pmatrix}
\begin{pmatrix}1&-x_i\\ -x_i& x_i^2\end{pmatrix}=\frac{(1-x_i^2)^2}{(1+x_i^2)^2}.
$$
Hence
$$
\mu_i=\frac{1-x_i^2}{1+x_i^2}, \quad i=1,\ldots, m.
$$
These numbers are the cosine of the  principal angles between the eigenspaces of $S$, hence
$$
\alpha_i^{\uparrow} = \arccos \frac{1-x_i^2}{1+x_i^2}=\arcsin \frac{2x_i}{1+x_i^2} \quad i=1,\ldots, m,
$$
establishing the first  formulae of our theorem.

Now, for $\alpha\in[0,\pi/2]$, the relation
$$
\arccos\frac{1-x^2}{1+x^2}=\alpha
$$
entails
$$
(1+x^2)\cos \alpha =1-x^2
$$
so
$$
x^2(1+\cos \alpha)=1-\cos \alpha.
$$
Therefore
$$
x^2=\frac{1-\cos\alpha}{1+\cos \alpha}=\frac{2\sin^2(\alpha/2)}{2\cos^2(\alpha/2)}.
$$
Consequently, for $0<x\le 1$, we have $x=\tan(\alpha/2)$ and
$\alpha=2\arctan x$,
establishing the second formula.

Now we check the last formula. With $a=x^{-1}-x$ and $x=\tan(\alpha/2)$ we obtain
$$
a=\frac{\cos(\alpha/2)}{\sin(\alpha/2)} -\frac{\sin(\alpha/2)}{\cos(\alpha/2)} =
\frac{\cos^2(\alpha/2)-\sin^2(\alpha/2)}{\sin(\alpha/2)\cos(\alpha/2)}=2\cotan\alpha
$$
and  the third formula follows.
\end{proof}

\vskip 5pt
\begin{remark} We note that our result for involutions  also provides the notion of principal angles between two subspaces $\trE$ and $\trF$. Indeed, without loss of generality we may assume that these subspaces are complementary.
Considering an involution with eigenspaces $\trE$ and $\trF$ and applying the theorem, we then reach the joint block diagonalization \eqref{jd3} for the orthogonal projections on $\trE$ and $\trF$. Hence, in the general case, we obtain \eqref{jd1}-\eqref{jd2}, that is, the existence of principal angles, without Proposition \ref{prop1}. 
\end{remark}

\vskip 5pt
Statements (a) and (b) of the following corollary were first noted by Ikramov \cite{Ik}. Statement (c) involves the Friedrichs angle, which is the smallest nonzero principal angle between two subspaces. Statement (d) follows from the elliptical range theorem, see \cite[Lemma 1.3.3, p.\ 20]{HJ} or \cite{Li}.

\vskip 5pt
\begin{cor}\label{cornum} Let $S$ be an involution with polar decomposition $S=U|S|$ and operator norm $s:=\| S\|_{\infty}$. Then :
\begin{itemize}
\item[(a)] $U$ is a Hermitian symmetry;
\item[(b)] $|S^*|=|S|^{-1}$;
\item[(c)] the sine of the Friedrichs angle between the eigenspaces of $S$ is $2/(s+s^{-1})$; 
\item[(d)] the numerical range $W(S)$ of $S$ is the elliptical disc with foci $\pm 1$, length of minor axis $s-s^{-1}$ and length of major axis $s+s^{-1}$. The minor axis lies on the imaginary axis and the major axis on the real axis.
\end{itemize}
\end{cor}

\vskip 5pt
Let us point here a very simple proof of (a) and (b). From $S=U|S|=|S^*|U$ and $S=S^{-1}$ we infer $U|S|=U^*|S^*|^{-1}$ and the result follows as the polar decomposition of an invertible matrix is unique.

The standard terminology for a reflection $R$ in $\bM_n$ is a unitary matrix $R$ such that
$\rk (I-R)=1$ and $\det R=1$. So, if one deletes the unitary assumption, one says that $R$ is a general reflection, or a skew reflection.  A skew reflection has only one principal angle between its two eigenspaces.

\vskip 5pt
\begin{cor} If $S\in\bM_n$ is a skew reflection, then the angle between its two eigenspaces is
$$
\alpha=\arccotan\frac{\sqrt{-n+{\mathrm{Tr\,}}S^*S}}{2}.
$$
\end{cor}

\vskip 5pt
\begin{proof} 
By Theorem \ref{th-inv}
 
$$S\simeq I_{n-2}\oplus
\begin{pmatrix}1&a\\ 0& -1\end{pmatrix}
$$
where $\alpha=\arccotan(a/2)$. Hence, denoting by $\|\cdot\|_2$ the Hilbert-Schmidt norm,
$$
\| S\|_2^2= n +a^2
$$
leading to the conclusion of the corollary.
\end{proof}

\vskip 5pt
\begin{cor} Let $E\in\bM_n$ be  a projection. Then  the principal angles between its range and its kernel satisfy 
$$
-{\mathrm{Tr\,}} E +  {\mathrm{Tr\,}} E^*E =\sum_{i=1}^m \cotan^2\alpha_i^{\uparrow}.
$$
\end{cor}

\vskip 5pt
\begin{proof} We consider the complementary projection $F=I-E$ and the involution $S=E-F=2E-I$. From Theorem \ref{th-inv}, taking the Hilbert-Scmidt norm for the second decomposition of $S$, we get
$$
{\mathrm{Tr\,}} S^*S= n+ 4\sum_{i=1}^m \cotan^2\alpha_i^{\uparrow}.
$$
Since
$$
{\mathrm{Tr\,}} S^*S={\mathrm{Tr\,}} (I- 2E+2E^* -4E^*E)=n -4{\mathrm{Tr\,}}E+4{\mathrm{Tr\,}}E^*E,
$$
the result follows. \end{proof}

\section{Numerical range and dilation}

\vskip 5pt
\begin{cor}\label{cordil}  Let $A\in\bM_n$ be a contraction. Then, there exists an involution $S\in\bM_{4n}$  with  $\|S\|_{\infty}=1+\sqrt{2}$ such that
$$
S=\begin{pmatrix} A&X \\ Y&Z\end{pmatrix}
$$
for some matrix $Z\in\bM_{3n}$ and some rectangular matrices $X,Y$ of suitable sizes. The norm $1+\sqrt{2}$ is the smallest possible one; a contraction  with $i$ in its spectrum cannot be dilated into an involution of norm $<1+\sqrt{2}$.
\end{cor}

\vskip 5pt
\begin{proof} We may dilate with Halmos $A$ into a unitary matrix in $\bM_{2n}$,
$$
\begin{pmatrix} A&-\sqrt{I-AA^*} \\
\sqrt{I-A^*A}&A^*
\end{pmatrix}
\simeq D:={\mathrm{diag}}(e^{i\theta_1},\ldots e^{i\theta_{2n}}).
$$
The involution 
$$
T=\begin{pmatrix} 0&1+\sqrt{2}\\
-1+\sqrt{2}&0
\end{pmatrix}
$$
has a numerical range $W(T)$ containing the unit disc, by (d) of Corollary \ref{cornum}. Hence the diagonal matrix  $D$ can be dilated into the involution $\oplus^{2n}T\in\bM_{4n}$.  For convenience of the reader we recall the proof of this classical fact. Pick unit vectors $u_1,\ldots, u_{2n}$ in $\bC^2$ such that
$$
\langle u_k, Tu_k\rangle =e^{i\theta_k}, \qquad k=1,\ldots, 2n,
$$
and define
\begin{align*}
v_1&=(u_1,0,\ldots, 0)\in (\bC^2)^{2n} =\bC^{4n} \\
&\ldots \\
v_{2n}&=(0,0,\ldots, u_{2n})\in\bC^{4n}.
\end{align*}
Then with respect to a basis of $\bC^{4n}$ starting with $v_1,v_2,\ldots, v_{2n}$, we see that $\oplus^{2n} T$ has a representation of the form
$$
\begin{pmatrix} D&K \\
L&M
\end{pmatrix}
$$
for some matrices $K,L,M\in\bM_{2n}$.

Therefore $A$ can be dilated into an involution $S\simeq\oplus^{2n}T$ and $\|S\|_{\infty}=1+\sqrt{2}$. This constant cannot be diminished as $i$ would not be any longer  in $W(S)$,  and since $W(S)\supset W(A)$, a contraction $A$ with $i$ in its spectrum
could not be dilated into $S$. \end{proof}

\vskip 5pt
In this proof, it seems necessary to use a complex matrix $S$ even if we start with a real matrix $A$. Hence, we ask the following.

\vskip 5pt
\begin{question} In the previous proposition, is it possible to take $S$ with only real entries when so is $A$ ?
\end{question}

\vskip 5pt
\begin{prop}\label{propdil} Every matrix $A\in\bM_n$ can be dilated into two involutions in $\bM_{2n}$, 
$$
S_A=
\begin{pmatrix} A& I+A\\
I-A&-A
\end{pmatrix}, \qquad 
T_A=
\begin{pmatrix} A& I-A^2\\
I&-A
\end{pmatrix}.
$$
If $A$ is a contraction, then $\|S_A\|_{\infty}\le 3$ and  $\|S_A\|_{2}\le \sqrt{6n}$.
\end{prop}

\vskip 5pt
\begin{proof}
Direct computation reveals that $S_A$ and $T_A$ are involutions. From
$$
S_A=\begin{pmatrix} 0&I \\I&0\end{pmatrix} + \begin{pmatrix} A&A \\ -A&-A\end{pmatrix} 
$$
we infer $\| S_A\|_{\infty} \le 1 +2\| A\|_{\infty}$, and so $\| S_A\|_{\infty}\le 3$ for any contraction $A$.
The Hilbert-Schmidt norm of $S_A$ satisfies
$$
\| S_A\|_2^2={\mathrm{Tr\,}} |I+A|^2 +{\mathrm{Tr\,}} |I-A|^2 + 2{\mathrm{Tr\,}} |A|^2.
$$
So 
$$
\| S_A\|_2^2= 4{\mathrm{Tr\,}} |A|^2 + 2{\mathrm{Tr\,}} I
$$
and, if $A$ is a contraction, this is maximized with $6n$.
\end{proof}

\vskip 5pt
\begin{question} The constant $\sqrt{6n}$ in Proposition \ref{propdil} is the best possible one. However, it seems possible to slighly diminish the constant $3$. What is the sharp value $\omega$ ? Is then $\omega$ the smallest possible operator norm  ensuring that any contraction in $\bM_n$ can be dilated into some matrix in $\bM_{2n}$
with norm less or equal than $\omega$ ?
\end{question}

\vskip 5pt
We may use in Proposition \ref{propdil}  the partial transpose 
$$
S_A^{\tau} =
\begin{pmatrix} A& I-A\\
I+A&-A
\end{pmatrix}, 
$$
and similarly for $T_A^{\tau}$. Further, if $A$ is invertible then we have the family parametrized by $x\in\bR$,
$$
S_A(x)= \begin{pmatrix} A& I+A^{x}\\
I-A^{2-x}&-A
\end{pmatrix}
$$
and the partial transposes $S_A^{\tau}(x)$.

\vskip 15pt
\noindent
Jean-Christophe Bourin

\noindent
Université Marie et Louis Pasteur, CNRS, LmB (UMR 6623), F-25000 Besançon, France.

\noindent

\noindent

\noindent
Email: jcbourin@univ-fcomte.fr

  \vskip 15pt\noindent
Eun-Young Lee

\noindent
 Department of mathematics, KNU-Center for Nonlinear
Dynamics,

\noindent
Kyungpook National University,

\noindent
 Daegu 702-701, Korea.

\noindent
 Email: eylee89@knu.ac.kr

\end{document}